\documentclass[secthm,seceqn,amsthm,ussrhead,12pt]{amsart}
\usepackage{amsmath,latexsym}
\usepackage[english]{babel}
\usepackage[psamsfonts]{amssymb}
\usepackage{times}
\usepackage{cite}

\usepackage[mathcal]{euscript}
\numberwithin{equation}{section} \textwidth=17.5cm
 \newtheorem{thm}{Theorem}[section]
 
 \newtheorem{lem}[thm]{Lemma}
 
 \theoremstyle{definition}
 
 \theoremstyle{remark}
 \newtheorem{rem}[thm]{Remark}
 
 \numberwithin{equation}{section}

\numberwithin{equation}{section}

\begin{document}

\title[$2$-Local derivations on von Neumann algebras] {$2$-Local derivations
 on von Neumann algebras}

\author{Shavkat Ayupov}

\address{%
Dormon yoli 29 \\
Institute of
 Mathematics  \\
 National University of
Uzbekistan\\
 100125  Tashkent\\   Uzbekistan\\
 and
 the Abdus Salam \\
 International Centre\\
 for Theoretical Physics (ICTP)\\
  Trieste, Italy}

\email{sh$_{-}$ayupov@mail.ru}

\author{Karimbergen Kudaybergenov}

\address{%
Ch. Abdirov 1 \\
Department of Mathematics\\
 Karakalpak state university \\
 Nukus
230113, Uzbekistan}

 \email{karim2006@mail.ru}

\begin{abstract} The paper is devoted to the description of $2$-local derivations
on von Neumann algebras. Earlier it was proved that every
$2$-local  derivation on a semi-finite von Neumann algebra is a
derivation. In this paper, using the analogue of Gleason Theorem
for signed measures, we extend this result to type $III$ von
Neumann algebras. This implies that on arbitrary
 von Neumann algebra each $2$-local derivation is a derivation.

\\

\end{abstract}

\maketitle

\section{Introduction}

Given an algebra $\mathcal{A},$ a linear operator
$D:\mathcal{A}\rightarrow \mathcal{A}$ is called a
\textit{derivation}, if $D(xy)=D(x)y+xD(y)$ for all $x, y\in
\mathcal{A}$ (the Leibniz rule). Each element $a\in \mathcal{A}$
implements a derivation $D_a$ on $\mathcal{A}$ defined as
$D_a(x)=[a, x]=ax-xa,$ $x\in \mathcal{A}.$ Such derivations $D_a$
are said to be \textit{inner derivations}.  Recall that a  map
$\Delta:\mathcal{A}\rightarrow\mathcal{A}$  (not linear in
general) is called a
 $2$-\emph{local derivation} if  for every $x, y\in \mathcal{A},$  there exists
 a derivation $D_{x, y}:\mathcal{A}\rightarrow\mathcal{A}$
such that $\Delta(x)=D_{x, y}(x)$  and $\Delta(y)=D_{x, y}(y).$

The notion of $2$-local derivations it was  introduced in 1997 by
P. \v{S}emrl \cite{S} and in this paper he described $2$-local
derivations on the algebra $B(H)$ of all bounded linear operators
on the infinite-dimensional separable Hilbert space $H.$ A similar
description for the finite-dimensional case appeared later in
\cite{KK}. In the paper \cite{LW} $2$-local derivations have been
described on matrix algebras over finite-dimensional division
rings.

In \cite{AK} we suggested a new technique and have generalized the
above mentioned results of \cite{KK} and \cite{S} for arbitrary
Hilbert spaces. Namely we considered $2$-local derivations on the
algebra $B(H)$ of all  bounded linear operators on an arbitrary (no
separability is assumed) Hilbert space $H$ and proved that every
$2$-local derivation on $B(H)$ is a derivation. A similar result
for $2$-local derivations on  finite von Neumann algebras  was obtained in
\cite{AKNA}. Finally, in \cite{AA2} the authors extended all above results
and give a short proof of this result for arbitrary semi-finite
von Neumann algebras.

In the present paper we prove that for  any purely infinite von
Neumann algebra $M$ every $2$-local derivation on $M$ is a
derivation. This completes the solution of the above problem for
arbitrary von Neumann algebras. Our proof is essentially based on
the analogue of Gleason theorem for signed measures on projection
of von Neumann algebras.

\medskip

\section{Main result}

\medskip

It is known \cite{Sak} that any derivation $D$ on a von Neumann
algebra $M$ is an inner derivation, that is there exists an
element $a\in M$ such that
$$
D(x)=D_a(x) = ax-xa
$$
for all $x\in M.$ Therefore for a von Neumann algebra $M$ the
above definition of $2$-local derivation is equivalent to the
following one: A map $\Delta : M\rightarrow M$ is called a
$2$-local derivation, if for any two elements $x$, $y\in M$ there
exists an element $a\in M$ such that
$$
\Delta(x)=ax-xa,\, \Delta (y)=ay-ya.
$$
If  $\Delta :M\rightarrow M$ is  a $2$-local derivation, then from
the definition it easily follows that $\Delta$ is homogenous. At
the same time,
\begin{equation}\label{joor}
 \Delta(x^2)=\Delta(x)x+x\Delta(x)
\end{equation}
for each $x\in M.$

In \cite{Bre} it is proved that any Jordan derivation (i.e. a linear map satisfying
the above equation) on a semi-prime algebra is a derivation.
Since every von Neumann algebra  is semi-prime, in the
case of von Neumann algebras in order  to prove that a $2$-local derivation
$\Delta :M\rightarrow M$ is a derivation it is sufficient to prove
that $\Delta:M\rightarrow M$ is additive.

The following theorem is the main result of the paper.

\begin{thm}\label{main}
Let $M$ be an arbitrary von Neumann algebra. Then any $2$-local
derivation $\Delta: M\rightarrow M$ is a derivation.
\end{thm}

For the proof of this theorem we need several lemmata.

For a self-adjoint subset $S\subseteq M$ denote by $S'$ is the
commutant of $S,$ i.e.
$$
S'=\{y\in B(H): xy=yx, \forall\, x\in S\}.
$$

\begin{lem}\label{masas}
Let $g\in M$ be a self-adjoint element and let
$\mathcal{W}^\ast(g)=\{g\}''$ be the abelian von Neumann algebra
generated by the element $g.$ Then  there exists an element $a\in M$
such that
\begin{equation*}\label{spat}
\Delta(x)=ax-xa
\end{equation*}
for all $x\in \mathcal{W}^\ast(g).$  In particular, $\Delta$ is
additive on $\mathcal{W}^\ast(g).$
\end{lem}

\begin{proof}  By the definition there exists an element $a\in M$
(depending on $g$) such that
$$
\Delta(g)=ag-ga.
$$
Let us show that $\Delta(x)=[a, x]$ for all $x\in
\mathcal{W}^\ast(g).$ Let $x\in \mathcal{W}^\ast(g)$ be an
arbitrary element. There exists an element $b\in M$ such that
$$
\Delta(g)=[b, g],\,  \Delta(x)=[b, x].
$$
Since
$$
[a, g]=\Delta(g)=[b, g],
$$
 we get
 $$
 (b-a)g=g(b-a).
 $$
  Thus
$$
b-a\in \{g\}'=\{g\}'''=\mathcal{W}^\ast(g)',
$$
i.e. $b-a$ commutes with any element from $\mathcal{W}^\ast(g).$
Therefore
$$
\Delta(x)=[b, x]=[b-a, x]+[a, x]=[a, x]
$$
for all $x\in \mathcal{W}^\ast(g).$ The proof is complete.
\end{proof}

In the following two lemmata we suppose that  $M$ is an \textit {infinite} von Neumann algebra.

\begin{lem}\label{addi}
The restriction
 $\Delta|_{M_{sa}}$  of the $2$-local
 derivation $\Delta$ onto the set $M_{sa}$ of all self-adjoint
 of $M$ is additive.
\end{lem}

\begin{proof}

Let $P(M)$ denote the lattice of all projections of the von
Neumann algebra $M$. Recall that a map $m:P(M)\rightarrow
\mathbb{C}$ is called \textit{a signed measure (or charge)} if
$m(e_1 +e_2) = m(e_1) + m(e_2)$ for arbitrary orthogonal
projections $e_1, e_2$ in $M$. A signed measure $m$ is said to be
\textit{bounded} if $\sup\{|m(e)|: e\in P(M)\}$ is finite.

Firstly, let us  show that for each  $f\in M_\ast$ the restriction
$f\circ \Delta|_{P(M)}$ of the superposition $f\circ \Delta(x) = f(\Delta(x)), x\in M$, onto the lattice $P(M)$ is a bounded
signed measure, where  $M_\ast$ is the predual space of  $M$ identified with the space of all normal
linear functionals on the von Neumann algebra $M.$

Let  $e_1, e_2$ be orthogonal projections in  $M.$ Put
$g=e_1+\frac{\textstyle 1}{\textstyle 2}e_2.$ By Lemma~\ref{masas}
there exists an element $a\in M$ such that  $\Delta(x)=ax-xa$ for
all $x\in \mathcal{W}^\ast(g).$ Since  $e_1, e_2\in
\mathcal{W}^\ast(g),$ we obtain that
$$
f\circ \Delta(e_1+e_2)=f([a, e_1+e_2])=f([a, e_1])+f([a,
e_2])=f\circ \Delta(e_1)+f\circ \Delta(e_2),
$$
i.e.  $ f\circ \Delta$ is a signed measure. Taking into account
that  $M$ is infinite,  from \cite[Theorem 29.5 and Theorem 30.08]{She} we obtain that the
signed measure $f\circ \Delta$ is bounded.

Now we can show that  $\Delta|_{M_{sa}}$ is additive.

Since  $f\circ \Delta|_{P(M)}$ is a bounded signed measure on
$P(M),$ by Gleason Theorem for signed measures (see \cite[Theorem
30.10]{She},  \cite[Theorem B]{Bun}) there exists a unique bounded
linear functional  $\widetilde{f}$ on $M$ such that
$\widetilde{f}|_{P(M)}=f\circ \Delta|_{P(M)}.$ Let us show that
$\widetilde{f}|_ {M_{sa}}=f\circ \Delta|_{M_{sa}}.$ Take an
arbitrary element  $x\in M_{sa}.$  By Lemma~\ref{masas} there
exists an element $a\in M$ such that  $\Delta(y)=ay-ya$ for all
$y\in \mathcal{W}^\ast(x).$ In particular,  $\Delta$ is linear on
 $\mathcal{W}^\ast(x),$ and therefore  $f\circ
\Delta|_{\mathcal{W}^\ast(x)}$ is a bounded linear functional
which is an extension of the measure  $f\circ
\Delta|_{P(\mathcal{W}^\ast(x))}.$ By the uniqueness of the extension
we have $\widetilde{f}(x)=f\circ\Delta(x).$ So
$f\circ\Delta|_{M_{sa}}$ is a bounded linear functional for all
$f\in M_\ast.$ This means, in particular, that
$$
f(\Delta(x+y)) = f(\Delta(x)) + f(\Delta(y)) = f(\Delta(x) +
\Delta(y)),
$$
 i.e.
 $f(\Delta(x+y) - \Delta(x)-\Delta(y)) = 0$ for
all $f\in M_ \ast$.
 Since  $M_\ast$ separates points on $M$ it follows that
  $\Delta(x+y) - \Delta(x)-\Delta(y) = 0$ for all $x,y \in M_{sa}$,
i.e.  $\Delta|_{M_{sa}}$  is additive. The proof is complete.
\end{proof}

\begin{lem}\label{jor} There exists an element  $a\in M$  such that
 $\Delta(x)=D_a(x) = ax-xa$ for all
$x\in M_{sa}.$
\end{lem}

\begin{proof}
Consider the extension $\widetilde{\Delta}$ of $\Delta|_{M_{sa}}$
on $M$ defined by:
$$
\widetilde{\Delta}(x_1+ix_2)=\Delta(x_1)+i\Delta(x_2),\, x_1,
x_2\in M_{sa}.
$$
Taking into account the homogeneity of $\Delta,$ Lemma~\ref{addi} and
the equality~\eqref{joor} we obtain that $\widetilde{\Delta}$ is a
Jordan derivation on $M$. As we mentioned above  by\cite[Theorem 1]{Bre}   any
Jordan derivation on a semi-prime algebra is a derivation. Since $M$ is
semi-prime $\widetilde{\Delta}$ is a derivation on $M$.
Therefore there exists an element $a\in M$ such that
$\widetilde{\Delta}(x)=ax-xa$ for all $x\in M.$ In particular,
$\Delta(x)=D_a(x) = ax-xa$ for all $x\in M_{sa}.$ The proof is complete.
\end{proof}

Now our aim is to show that if two $2$-local derivations coincide on $M_{sa}$
then they are equal on the whole von Neumann algebra $M$.

Further in Lemmata~\ref{supp}-\ref{inve}  we assume that $\Delta$
is a 2-local derivation on $M$  such that $\Delta|_{M_{sa}}\equiv
0.$

\begin{lem}\label{supp} Let $x\in M$ and let $e, f$ be projections
in $M$ such that $x=exf.$ Then
 \begin{enumerate}
\item[(i)] $e\Delta(x)f=\Delta(x).$
\item[(ii)] if    $ef=0,$ then $\Delta(x)=0.$
 \end{enumerate}
\end{lem}

\begin{proof} $(i)$ Take a derivation $D$ on $M$ such that
$$
\Delta(x)=D(x),\, \Delta(e)=D(e).
$$
Taking into account  $D(e)=\Delta(e)=0$ and $x=ex$  we obtain
$$
\Delta(x)=D(x)=D(ex)=D(e)x+eD(x)=eD(x)=e\Delta(x).
$$
In a similar way  $\Delta(x)=\Delta(x)f.$ Therefore $e\Delta(x)f=\Delta(x).$

$(ii)$ Set $y=exf+fx^\ast e.$ Take an element $a\in M$ such that
$$
\Delta(x)=[a, x],\, \Delta(y)=[a, y].
$$
Since $y$ is self-adjoint, we have $\Delta(y)=0.$ Taking into account the assumption
$exf=x$ by straightforward calculations we obtain that $e[a, y]f=e[a, x]f$. Thus
$$
0=e\Delta(y)f=e[a, y]f=e[a, x]f=e\Delta(x)f=\Delta(x).
$$
 The proof is complete.
\end{proof}

In the next four lemmata $p, q$ are mutually orthogonal equivalent
projections in $M$ with $p+q=\mathbf{1}.$

\begin{lem}\label{dia} If $x=pxq+qxp,$ then
 $\Delta(x)=0.$
 \end{lem}

\begin{proof} Take a derivation $D$ on $M$ such that
$$
\Delta(x)=D(x),\, \Delta(p)=D(p).
$$
Taking into account $D(p)=\Delta(p)=0$ and $pxp=0$ we have
$$
p\Delta(x)p=pD(x)p=D(pxp)-D(p)xp-xpD(p)=D(0)=0,
$$
i.e. $p\Delta(x)p=0.$
 Likewise $q\Delta(x)q=0.$

Now we set $y=pxq.$ Take an element $a\in M$ such that
$$
\Delta(x)=[a, x],\, \Delta(y)=[a, y].
$$
By Lemma~\ref{supp}~(ii) we have $\Delta(y)=0.$ Further
$$
p\Delta(x)q=p[a, x]q=p[a, y]q=p\Delta(y)q=0,
$$
i.e.
$$
p\Delta(x)q=0.
$$
Similarly $q\Delta(x)p=0.$ Therefore
$$
\Delta(x)=(p+q)\Delta(x)(p+q)=p\Delta(x)p+p\Delta(x)q+q\Delta(x)p+q\Delta(x)q=0.
$$
The proof is complete.
\end{proof}

\begin{lem}\label{fidia} If $x=pxp,$  then
 $\Delta(x)=0.$
 \end{lem}

\begin{proof} Take a partial isometry $u$ in $M$ such that
$$
u^\ast u=p,\, uu^\ast =q.
$$
 Set $y=xu^\ast +u.$ Since
$$
xu^\ast +u=pxu^\ast q +qup,
$$
we have that  $y=pyq+qyp$ and
Lemma~\ref{dia} implies that  $\Delta(y)=0.$
 Thus
$\Delta(y^2)=\Delta(y)y+y\Delta(y)=0.$ Take an element $a\in M$
such that
$$
\Delta(x)=[a, x],\, \Delta(y^2)=[a, y^2].
$$
It is easy to see that
$$
y^2=(xu^\ast +u)^2=(pxu^\ast q +qup)^2=x+uxu^\ast=pxp+quxu^\ast
q,
$$
therefore $p[a,y^2]p=p[a,x]p$. Now from Lemma~\ref{supp}~(i) we
obtain
$$
\Delta(x)=p\Delta(x)p=p[a, x]p=p[a, y^2]p=p\Delta(y^2)p=0,
$$
i.e.
$$
\Delta(x)=0.
$$
 The proof is complete.
\end{proof}

\begin{lem}\label{inv} Let  $x$ be an element such that $x=pxp$ and let $a\in M$ be an element
such that $\Delta(x)=[a, x].$ If $x$ is invertible in $pMp$  then
 $a=pap+q a q.$
 \end{lem}

\begin{proof} By Lemma~\ref{fidia} $\Delta(x)=0,$ and therefore  $p[a,
x]q=0.$  Replacing $x$ by $pxp$ we obtain $pxppaq=0.$ Since
$x=pxp$ is invertible in $pMp$ , i.e. $bpxp =p$ for an appropriate
$b \in pMp$, multiplying the last equality from the left side by
this $b$ we have $paq=0.$ Similarly $q ap=0.$ Thus
$$
a=(p+q)a(p+q)=pap+qaq.
$$ The proof is complete. \end{proof}

\begin{lem}\label{inve}
 Let $x$ be an arbitrary element in $M$. Then $p\Delta(x)p=0.$
 \end{lem}

\begin{proof} Since $\Delta$ is homogeneous, we can assume that
$\| x\|<1.$ Take a derivation $D$ on  $M$ such that
$$
\Delta(x)=D(x),\, \Delta(x-\mathbf{1})=D(x-\mathbf{1}).
$$
Then
$$
\Delta(x-\mathbf{1})=D(x-\mathbf{1})=D(x)-D(\mathbf{1})=
D(x)=\Delta(x),
$$
i.e.
$$
\Delta(x-\mathbf{1})=\Delta(x).
$$
Thus  replacing, if necessary, $x$ by  $x-\mathbf{1},$ we may assume
that $x$ is invertible. In particular, $pxp$ is also invertible in
$pMp.$

Now take an element $a\in M$ such that
$$
\Delta(x)=[a, x],\, \Delta(pxp)=[a, pxp].
$$
Since $pxp$ is  invertible in $pMp,$ by Lemma~\ref{inv} we have
$a=pap+q a q.$ Further
$$
p\Delta(x)p=p[a, x]p=p[a, pxp]p=p\Delta(pxp)p=0,
$$
where the last equality follows from Lemma~\ref{fidia} applied to the element $pxp$.
Therefore
$$
p\Delta(x)p=0.
$$
  The proof is
complete.
\end{proof}

\begin{lem}\label{nul}
 Let $M$ be a type $III$ von Neumann algebra and let
$y\in M$ be an arbitrary element. If  $pyp=0$ for all
projection with $p\sim \textbf{1}-p,$ then $y=0.$
 \end{lem}

\begin{proof} Let $y= y_1+i y_2$ be the decomposition of $y$ into
self-adjoint components and let $p$ be a projection with $p\sim
\textbf{1}-p.$ Then $py_1p+ipy_2p=0.$ Passing to the adjoint in
this equality we obtain $py_1p-ipy_2p=0.$ Thus $py_1p=0$ and
$py_2p=0.$ Therefore it is suffices to consider  the case where $y$ is a
self-adjoint element.

Let us show that $pyp=0$ for all projection with $p\preceq
\textbf{1}-p.$  Take an arbitrary projection $f\leq \mathbf{1}-p$
with $f\sim p.$ If $f+p=\mathbf{1}$ then $p\sim\mathbf{1}-p,$ and
therefore $pyp=0.$  Suppose that $f+p\neq \mathbf{1}.$ Since $M$
is a type $III$ von Neumann algebra, $\mathbf{1}-f-p$ is infinite.
Therefore there exist mutually orthogonal equivalent projections
$f_1, f_2$ such that $f_1+f_2=\mathbf{1}-f-p.$ In this case we put
$e=p+f_1.$ Then it is clear that $e\sim \mathbf{1}-e,$ and by the assumption of the
lemma $eye=0.$ Multiplying from the both sides by $p$ we obtain
that $pyp=0.$

Now suppose that $y\neq0$ and consider the support $s(y)$ of the
element  $y,$ i.e.
$$
s(y)=\mathbf{1}-\sup\{e\in P(M): ey=0\}.
$$
Let $A$ be an arbitrary maximal abelian $\ast$-subalgebra in
$s(y)Ms(y)$ containing the self-adjoint element $y.$
It is known (see\cite[Problem
6.9.19]{KR}) that if $\mathcal{R}$ is a von Neumann algebra with no abelian
central summands, and  $A$ is a maximal abelian
$\ast$- subalgebra of $\mathcal{R},$ then
$A$ contains a projection $e$ such that $c(e) =
c(\mathbf{1}-e)=\mathbf{1}$ and $e\preceq \mathbf{1}-e,$ where $c(e)$ denotes the central cover of the
projection $e$.
Applying this assertion to the von Neumann algebra $\mathcal{R}= s(y)Ms(y)$ we obtain that
 there exists a non zero projection $p$ in $A$ such
that $p\preceq s(y)-p\leq\mathbf{1}-p.$ Then from above we have $ py=pyp=0. $ This
contradicts with $0\neq p\leq s(y).$ From this contradiction we
obtain that $y=0.$
 The proof is complete.
\end{proof}

Since the projection $p$  (assumed that $p\sim q=\mathbf{1} - p$ )
in Lemma~\ref{inve} is arbitrary, from Lemma \ref{nul} we obtain
the following result.

\begin{lem}\label{jorr}
Let $M$ be a type $III$ von Neumann algebra. If
$\Delta|_{M_{sa}}\equiv 0$ then  $\Delta\equiv 0.$
\end{lem}

\begin{rem} \label{rem}
Let $e$ be a central projection in $M$. Since $D(e)=0$ for an
arbitrary derivation $D$, it is clear that $\Delta(e)=0$ for any
$2$-local derivation $\Delta$ on $M$. Take $x\in M$ and let $D$ be
a derivation on $M$ such that $\Delta(ex)=D(ex), \Delta(x)=D(x)$.
Then we have $\Delta(ex)=D(ex)=D(e)x+eD(x)=e\Delta(x)$. This means
that every $2$-local derivation $\Delta$ maps $eM$ into $eM$ for
each central projection $e\in M$, i.e. if necessary, we may
consider the restriction of $\Delta$ onto $eM$.
\end{rem}

\textit{Proof of the Theorem \ref{main}}. By
\cite[Theorem]{AA2} any $2$-local derivation on a semi-finite von
Neumann algebra is a derivation, therefore the above Remark implies that it is
sufficient to consider only a type $III$ von Neumann algebra $M$. By
Lemma~\ref{jor} there exists an element $a\in M$ such that
 $\Delta(x)=D_a(x)= ax-xa$ for all
$x\in M_{sa}.$ Consider the $2$-local derivation $\Delta-D_a.$ Since
$(\Delta-D_a)|_{M_{sa}}\equiv 0,$ Lemma~\ref{jorr} implies that
$\Delta=D_a.$ The proof is complete.

\end{document}